\documentclass{amsart}
\usepackage{amssymb,amsmath,bbm}
\usepackage{graphicx} 
\usepackage{caption}

\newcommand{\RR}{\mathbb{R}}
\newcommand{\QQ}{\mathbb{Q}}

\newcommand{\PSD}{\mathcal{S}_+}
\newcommand{\rank}{\textup{rank}\,}

\newcommand{\rankpsd}{\textup{rank}_{\textup{psd}}\,}

\newcommand{\conv}{\textup{conv}}
\newcommand{\cone}{\textup{cone}}

\newcommand{\onevec}{\mathbbm{1}}

\newtheorem{theorem}{Theorem}[section]

\newtheorem{lemma}[theorem]{Lemma}

\newtheorem{proposition}[theorem]{Proposition}
\theoremstyle{definition}
\newtheorem{definition}[theorem]{Definition}
\newtheorem{example}[theorem]{Example}

\theoremstyle{remark}

\title{Worst-Case Results for Positive Semidefinite Rank}

\author{Jo{\~a}o Gouveia}
\address{CMUC, Department of Mathematics,
  University of Coimbra, 3001-454 Coimbra, Portugal}
\email{jgouveia@mat.uc.pt} 

\author{Richard Z. Robinson}
\address{Department of Mathematics, University of Washington, Box
  354350, Seattle, WA 98195, USA} \email{rzr@uw.edu}

\author{Rekha R. Thomas}
\address{Department of Mathematics, University of Washington, Box
  354350, Seattle, WA 98195, USA} \email{rrthomas@uw.edu}

\thanks{Gouveia was supported by by the Centre for Mathematics at the University of Coimbra and Fundac\~ao para a Ci\^encia e a
Tecnologia, through the European program COMPETE/FEDER, Robinson by 
the U.S. National Science Foundation Graduate Research Fellowship under Grant No. DGE-0718124, and Thomas by 
the U.S. National Science Foundation grant DMS-1115293.}  

\begin{document}

\begin{abstract}
We present various worst-case results on the positive semidefinite (psd) rank of a nonnegative matrix, primarily in the context of polytopes. We prove that the psd rank of a generic $n$-dimensional polytope with $v$ vertices is at least 
$(nv)^{\frac{1}{4}}$ improving on previous lower bounds. For polygons with $v$ vertices, we show that psd rank cannot exceed $4 \left \lceil v/6 \right \rceil$ which in turn shows that the psd rank of a $p \times q$ matrix of rank three is at most $4\left\lceil \min\{p,q\}/6 \right\rceil$. In general, a nonnegative matrix of rank ${k+1 \choose 2}$  has psd rank at least $k$ and we pose the problem of deciding whether the psd rank is exactly $k$.  
Using geometry and bounds on quantifier elimination, we show that this decision can be made in polynomial time when $k$ is fixed. 
\end{abstract}

\maketitle

\section{Introduction}
The {\em positive semidefinite (psd) rank} of a nonnegative matrix was introduced in \cite{FMPTW} and \cite{GPT2011} and is a special case of the {\em cone rank} of a nonnegative matrix from \cite{GPT2011}.  A familiar example of a cone rank is that of {\em nonnegative rank}; given a nonnegative matrix 
$M \in \RR^{p \times q}$, its nonnegative rank is the smallest positive integer $k$ such that there exists nonnegative vectors $a_1, \ldots, a_p, b_1,\ldots,b_q \in \RR^k_+$ such that $M_{ij} = a_i^Tb_j$ for each $i$ and $j$. Let $\mathcal{S}^k$ denote the vector space of all real symmetric $k \times k$ matrices and $\PSD^k$ denote the cone of all psd matrices in $\mathcal{S}^k$. Then the psd rank of $M$ is the smallest integer $k$ such that there exists matrices $A_1,\ldots,A_p, B_1, \ldots, B_q \in \PSD^k$ such that $M_{ij} = \langle A_i,B_j\rangle := \textup{Trace}(A_iB_j)$ for all $i$ and $j$. We denote it as $\rankpsd(M)$ and call the matrices $A_1,\ldots,B_q$ a \emph{psd factorization} of $M$.

Nonnegative rank has been studied for several years, has many applications \cite{CohenRothblum}, and is NP-hard to compute 
\cite{Vavasis}. There are several techniques for finding lower bounds to nonnegative rank, most of which are based on the zero/nonzero structure (support) of the matrix \cite{CohenRothblum,FioriniKaibelPashkovichTheis}, and some newer methods that do not rely on support \cite{BraunFioriniPokuttaSteurer,FawziParrilo}. For a nonnegative $p \times q$ matrix, the psd rank is at most the nonnegative rank which in turn is at most $\min\{p,q\}$.  However, psd rank is more complicated than nonnegative rank with almost no techniques available for finding bounds. In this paper we exhibit bounds for psd rank in several contexts, and establish new tools for studying this rank. 

The motivation for the definition of psd rank came from geometric problems concerning the representation of convex sets for linear optimization. For instance, given a polytope $P \subset \RR^n$, one can ask whether $P$ can be expressed as the linear image of an affine slice of some positive orthant $\RR^k_+$. If $k$ is small relative to the number of facets of $P$, then this implicit representation of $P$ can be far more efficient that the natural representation of $P$ by inequalities in $n$ variables. In the foundational paper \cite{Yannakakis}, Yannakakis proved that the smallest $k$ possible for a given $P$ is the nonnegative rank of a {\em slack matrix} of $P$ (also called 
the nonnegative rank of $P$). If $P$ is a full-dimensional polytope in $\RR^n$ with vertices $p_1, \ldots, p_v$ and $f$ facet-defining inequalities $d_j - c_j^Tx \geq 0$ where $c_j \in \RR^n, d_j \in \RR$, then the corresponding slack matrix of $P$ is the $v \times f$ nonnegative matrix whose $(i,j)$-entry is $d_j - c_j^Tp_i$, the {\em slack} of vertex $p_i$ in the facet inequality $d_j - c_j^Tx \geq 0$. Yannakakis' result was extended in \cite{FMPTW} and \cite{GPT2011} to show that the psd rank of a slack matrix of $P$ (often referred to as the psd rank of $P$) is exactly the smallest $k$ such that $P$ is the linear image of an affine slice of the psd cone $\PSD^k$.  Affine slices of the psd cone are called \emph{spectrahedra} and can be written as $\left\{ (x_1,\ldots,x_{d-1})\; | \; g(x_1,\ldots,x_{d-1})\succeq 0 \right\}$ where $g$ is the linear matrix pencil $x_1G_1+\ldots+x_{d-1}G_{d-1} + G_d$ defined by the matrices $G_1,\ldots,G_d \in \mathcal{S}^k$.  Again, if $k$ is small compared to $f$, one can very often optimize a linear function  efficiently over $P$ via {\em semidefinite programming}. This geometric connection has made psd rank an important invariant of a polytope and both upper and lower bounds on this rank shed information on the complexity of the polytope. One can further extend Yannakakis' theorem (and hence psd rank) to all convex sets \cite{GPT2011} but in this paper we only consider polytopes. 

Finding lower bounds on the psd rank of a nonnegative matrix is notoriously hard, even harder than in the case of nonnegative rank, and only somewhat trivial bounds are known. For example, dimension counting is enough to conclude that the psd rank of a matrix $M$ is at least $\frac{1}{2} \sqrt{1+8 \, \textup{rank}(M)}-\frac{1}{2}$.  Support based bounds, the most popular type of bounds for nonnegative rank, are of limited strength in the psd case, but can
still yield interesting applications, as shown in \cite{LeeTheis}. In the special case of slack matrices of polytopes, we can do slightly better. For example, if $M$ is the slack matrix of an $n$-dimensional polytope $P$, we actually have $\rankpsd(M) \geq \rank(M) = n+1$, as seen in  \cite{GRT2012,LeeTheis}, by support based arguments. Furthermore, direct application of quantifier elimination bounds guarantees that if $M$ has psd  rank $k$, then $P$ has at most  $k^{O(k^2n)}$ facets (see \cite{GPT2011}). This translates to saying that $\rankpsd(P) \geq \textup{Exp}(\frac{1}{2} W(O(\log(f)/n)) )$, where $f$ is the number of facets of $P$ and $W$ the Lambert $W$-function.  Using the asymptotic behavior of $W$, this results in a lower bound of the type $\Omega \left(\sqrt{\frac{\log(f)}{n\log\log(f)}}\,\right)$.

In this paper we establish several new bounds on psd rank for particular families of matrices with a focus on slack matrices of polytopes. In Section~2 we show that a generic $n$-dimensional polytope with $v$ vertices has psd rank at least $(nv)^{\frac{1}{4}}$, much improving the lower bounds discussed earlier. An analogous result for nonnegative rank of generic polytopes was proven in \cite{FioriniRothvossTiwary}. This implies that a generic polygon with $v$ vertices has psd rank at least $(2v)^{\frac{1}{4}}$ while all $v$-gons have psd rank at most $v$. In Section~3 we improve the upper bound for polygons from $v$ to $4 \left\lceil v/6 \right \rceil$ by showing that all hexagons have psd rank four and then using some psd rank calculus from \cite{GPT2011}. Slack matrices of polygons have rank three and we use the previous upper bound on polygons to show that all rank three matrices have psd rank at most $4\left\lceil \min\{p,q\}/6 \right\rceil$. These results are psd analogs of results on nonnegative rank in \cite{Shitov}, where it is shown that a $v$-gon has nonnegative rank at most $\left\lceil \frac{6v}{7} \right\rceil$. Next we shift gears in Section~4 and examine how low the psd rank of a matrix of fixed rank can get.  A nonnegative matrix of rank ${k+1 \choose 2}$ has psd rank at least $k$. This bound is tight if and only if it is possible to sandwich the psd cone $\PSD^k$ in between two polyhedral cones coming from $M$. We then reduce this geometric condition to the feasibility of a semialgebraic system and use results on quantifier elimination to show that when $k$ is fixed, it is possible to decide in polynomial time whether a nonnegative matrix of rank ${k+1 \choose 2}$ has psd rank $k$.

\section{A Lower Bound on PSD Rank of Generic Polytopes}

In this section we will focus on lower bounds for the psd ranks of generic polytopes. A polytope $P \subset \RR^n$ is said to be {\em generic} if the coordinates of its vertices form an algebraically independent set over the rationals, i.e. the vertex coordinates do not satisfy any non-trivial polynomial equation with rational coefficients. It is clear that no simple description of such a polytope can be expected. It was shown in \cite{FioriniRothvossTiwary} that the nonnegative rank of a generic polygon with $v$ vertices is at least $\sqrt{2v}$. Their proof in fact extends to showing that the nonnegative rank of a generic $n$-dimensional polytope with $v$ vertices is at least $\sqrt{nv}$. We adapt the philosophy of their proof to the psd case to prove the lower bound, $\rankpsd(P) \geq (nv)^{\frac{1}{4}}$.

\begin{theorem} \label{thm:generic polytope}
If $P \subset \RR^n$ is a generic polytope with $v$ vertices, then its psd rank is at least $(nv)^{\frac{1}{4}}$.
\end{theorem}

\begin{proof}
Let  $d_k := \frac{k(k+1)}{2}$ be the dimension of the $\RR$-vector space $\mathcal{S}^k$ of $k \times k$ real symmetric matrices. Suppose $P \subset \RR^n$ is a generic polytope with $v$ vertices and $\rankpsd (P)=k$.  Then $P$ is the image under a linear map of a spectrahedron living in $\PSD^k$.  Without loss of generality, we may assume that this linear map is the projection onto the first $n$ coordinates and that $P$ can be written as:
\[ P = \left\{ (x_1,\ldots,x_n)\, | \, \exists \; x_{n+1},\ldots,x_{d_k - 1} \textup{ with } g(x_1,\ldots,x_{d_k-1}) \succeq 0 \right\} \]
where $g(x_1,\ldots,x_{d_k-1}) = x_1G_1 + \ldots +x_{d_k-1}G_{d_k-1} + G_{d_k}$, each $G_i \in \mathcal{S}^k$.

Let $\Gamma$ be the set of distinct real entries in the matrices $G_i$.  Then $|\Gamma| \leq d_k^2 \leq k^4$.  Consider the extension field $\QQ(\Gamma)$ and its real closure $\overline{\QQ(\Gamma)}$ (this is simply the real part of the algebraic closure of $\QQ(\Gamma)$).  The transcendence degree of $\overline{\QQ(\Gamma)}$ is at most $|\Gamma|$ (see \cite[Chap~6]{Hungerford} for the definition of transcendence degree and its basic properties).  We now show that the vertex coordinates of $P$ are all contained in $\overline{\QQ(\Gamma)}$.

Let $p = (p_1,\ldots,p_n)$ be a vertex of $P$ and $\omega \in \QQ^n$ a vector such that the linear program 
($\mathcal{L}$): $\max \{ \omega^Tx \,:\, x \in P \}$
has $p$ as its unique optimal point.  Let $\widetilde{\omega} :=(\omega,0,\ldots,0) \in \QQ^{d_k-1}$ and $S :=\left\{ (x_1,\ldots,x_{d_k-1}) \, | \, g(x_1,\ldots,x_{d_k-1}) \succeq 0 \right\} $. Then the semidefinite program ($\mathcal{SP}$): $\max \{  \widetilde{\omega}^T x  \,:\, x\in S \}$ 
has the same optimal value as ($\mathcal{L}$) and the solutions of ($\mathcal{SP}$) are the points in $S$ that project to $p$.

If $S \cap \textup{int}(\PSD^k) = \emptyset$, then $S$ could be written as an affine slice of a proper face of $\PSD^k$.  Since proper faces of the psd cone are isomorphic to smaller psd cones \cite[Chap~3]{Wolkowicz}, this would mean that $P$ could be written as a projection of an affine slice of $\PSD^l$ for some $l < k$, which would contradict our original assumption of $\rankpsd (P) = k$.  Hence, we must have that $S \cap \textup{int}(\PSD^k) \neq \emptyset$. Thus, Slater's condition \cite[Chap~4]{Wolkowicz} is satisfied for ($\mathcal{SP}$).  Let ($\mathcal{SD}$) denote the semidefinite program dual to ($\mathcal{SP}$).  From semidefinite programming duality theory \cite[Chap~4]{Wolkowicz}, we know that a pair $(x,Y) \in \RR^{d_k-1} \times \mathcal{S}^k $ is an optimal primal-dual pair for the programs ($\mathcal{SP}$), ($\mathcal{SD}$) if and only if they satisfy the following first order conditions:
\[ g(x) \succeq 0,\,\,\,Y \succeq 0, \,\,\, \left\langle G_i,Y \right\rangle = -\widetilde{\omega}_i ,\,\,\, \left\langle g(x),Y \right\rangle = 0 . \]
These conditions are a series of polynomial equations and inequalities with coefficients in $\QQ(\Gamma)$.  By the Tarski-Seidenberg Theorem \cite[Chap~5]{bochnak}, we have that there exists a solution to these equations over $\RR$ if and only if there exists a solution to these equations over $\overline{\QQ(\Gamma)}$.  Since Slater's condition was satisfied, strong duality holds for $(\mathcal{SP})$, and there are solutions $(x,Y)$ over $\RR$.  By our choice of ($\mathcal{SP}$), we have that each of these solutions $(x,Y)$ is of the form $(p_1,\ldots,p_n,x_{n+1},\ldots,x_{d_k-1},Y)$.  Hence, the coordinates $p_1,\ldots,p_n$ are contained in $\overline{\QQ(\Gamma)}$.

By repeating this procedure for each vertex of $P$, we see that $\overline{\QQ(\Gamma)}$ contains all $n$ coordinates for each of the $v$ vertices.  Hence, $\overline{\QQ(\Gamma)}$ contains $nv$ algebraically independent elements.  Thus, the transcendence degree of $\overline{\QQ(\Gamma)}$ is at least $nv$.  Hence, we have that $nv \leq |\Gamma | \leq k^4$.
\end{proof}

In \cite{FioriniRothvossTiwary}, the authors also prove that for each $v \geq 3$, there is a $v$-gon with integer vertices lying in $[2v] \times [4v^2]$ whose nonnegative rank is $\Omega(({v}/{\textup{log } v})^{\frac{1}{2}})$. The same statement also holds in the psd setting with the bound changing to 
$\Omega(({v}/{\textup{log } v})^{\frac{1}{4}})$
as recently shown in \cite{BrietDadushPokutta}.

\section{An Upper Bound on PSD Rank of Polygons}

The result in the previous section implies that the psd rank of a generic $v$-gon is at least $(2v)^{\frac{1}{4}}$, while on the other hand, $v$ is a trivial upper bound on the psd rank of any $v$-gon since its slack matrix has size $v \times v$. This tells us that the worst case rank of a $v$-gon lies somewhere between the two. In this section we will use some simple geometric tools to show that the trivial upper bound can be slightly improved by a constant to $4\left\lceil \frac{v}{6} \right\rceil$. This result can be stated more generally for matrices of rank three. We will show that if $M$ is a nonnegative matrix of rank three of size $p \times q$, then $\rankpsd(M) \leq 4\left\lceil \frac{\min\{p,q\} }{6} \right\rceil$.  These results are analogous to recent results on nonnegative rank of polygons and rank three marices. In \cite{Shitov}, Shitov proved that the nonnegative rank of a $v$-gon is at most $\left\lceil \frac{6v}{7} \right\rceil$, and more generally, that the nonnegative rank of a rank three nonnegative matrix of size $p \times q$ is at most $\left\lceil \frac{6 \textup{min}\{p,q\}}{7} \right\rceil$. 

We begin with a general lemma about psd rank of polytopes.

\begin{lemma}\label{lem:adding a facet}
Let $P$ be a polytope with $\rankpsd(P)=k$, and let $\widetilde{P}$ be a polytope obtained from $P$ by adding either a single inequality to the facet description of $P$ or a single point to the vertex description of $P$.  Then $\rankpsd(\widetilde{P}) \leq k+1$.
\end{lemma}

\begin{proof}
First, suppose that $\widetilde{P}$ arises by adding a single inequality to the facet description of $P$.  Then there exists some affine halfspace $A$ such that $\widetilde{P}=P\cap A$.  Write $A$ in the form $\left\{ x \in \RR^n \; | \; a_0+a_1x_1+\ldots+a_nx_n \geq 0 \right\}$.  Since the psd rank of $P$ is $k$, we can write $P$ in the form:
\begin{equation}\label{eq:spect form}
P= \left\{ (x_1,\ldots,x_n)\, | \, \exists \; x_{n+1},\ldots,x_{d_k - 1} \textup{ with } g(x_1,\ldots,x_{d_k-1}) \succeq 0 \right\} 
\end{equation}
where $g$ is the linear pencil given by matrices $G_1,\ldots,G_{d_k} \in \mathcal{S}^k$.  
Now define a vector $\widetilde{a} \in \RR^{d_k}$ with $\widetilde{a}=(a_1,\ldots,a_n,0,\ldots,0,a_0)$ and define matrices $\widetilde{G}_i \in \mathcal{S}^{k+1}$ where the upper left block is $G_i$, the lower right diagonal entry is $\widetilde{a}_i$, and all other entries are $0$. 
If we let the $\widetilde{G}_i$'s play the role of the $G_i$'s in (\ref{eq:spect form}), then this new set will be equal to $\widetilde{P}$.  Hence, $\widetilde{P}$ has a lift into $\PSD^{k+1}$ and we have that $\rankpsd(\widetilde{P}) \leq k+1$.

The case when $\widetilde{P}$ arises by adding a point to the vertex description of $P$ follows from the fact that 
a polytope and its polar both have the same psd rank \cite{GPT2011}.
\end{proof}

\begin{example}\label{ex:pentagons}
By \cite[Theorem~4.7]{GRT2012}, all triangles and quadrilaterals have psd rank three and any polygon with at least five sides has psd rank at least four.  Since a pentagon can be obtained by adding an inequality to the facet description of a quadrilateral, Lemma~\ref{lem:adding a facet} implies that all pentagons have psd rank exactly four.
\end{example}

The following lemma is a direct consequence of the definition of psd rank.

\begin{lemma}\label{lem:projections}
Let $P$ be a polytope and suppose there exists a polyhedron $Q$ and a linear map $\pi$ such that $P=\pi(Q)$.  Then $\rankpsd(P) \leq \rankpsd(Q)$.
\end{lemma}

\begin{theorem}\label{thm:hexagons}
Every hexagon has psd rank exactly four.
\end{theorem}

\begin{proof}
Let $H$ be a hexagon.  We know that $\rankpsd(H) \geq 4$ \cite[Theorem~4.7]{GRT2012}.  Since psd rank is invariant under invertible affine transformations, we may assume that $H$ has vertices $(1,0)$, $(a,b)$, $(0,1)$, $(c,d)$, $(0,0)$, and $(e,f)$ where $(a,b)$, $(c,d)$, and $(e,f)$ lie in the first, second, and fourth quadrants, respectively, and these points also satisfy $a+b>1$, $c+d<1$, and $e+f<1$.

Consider the polytope $O$ in $\RR^3$ with vertices $(0,0,0)$, $(1,0,0)$, $(0,1,0)$, $(0,0,1)$, $(v_1,0,v_3)$, and $(0,w_2,w_3)$, where 
\[v_1 = c-\frac{ad}{b},\: v_3 = \frac{d}{b},\: w_2 = f - \frac{be}{a},\: w_3 = \frac{e}{a}. \]
With this choice of coordinates, we see that $v_1<0$, $v_3 > 0$, $w_2 < 0$, $w_3 > 0$, $v_1 + v_3 < 1$, and $w_2+w_3 < 1$.  These conditions imply that $O$ is a combinatorial octahedron.
In \cite{GRT2012}, an octahedron $O$ was defined to be \emph{biplanar} if there exist two distinct planes $E_1$ and $E_2$ such that $O \cap E_i$ contains four vertices of $O$ for $i=1,2$.  By intersecting the $O$ defined above with the xz and yz-planes, we see that it is biplanar.  Thus by \cite[Theorem~4.8]{GRT2012}, we have that $\rankpsd(O) = 4$.  Define a linear map $\pi: \RR^3 \rightarrow \RR^2$ by the matrix $\left( \begin{array}{ccc} 1 & 0 & a \\ 0 & 1 & b \end{array} \right)$.  Then $\pi(O)=H$ and by Lemma~\ref{lem:projections}, $\rankpsd(H) = 4$.  This lift of a hexagon to an octahedron is shown in Figure~\ref{fig:hexlift}.
\end{proof}

\begin{figure}[h!]
\ \ \ \
\begin{minipage}[c]{0.35\textwidth}
    \includegraphics[width=\linewidth]{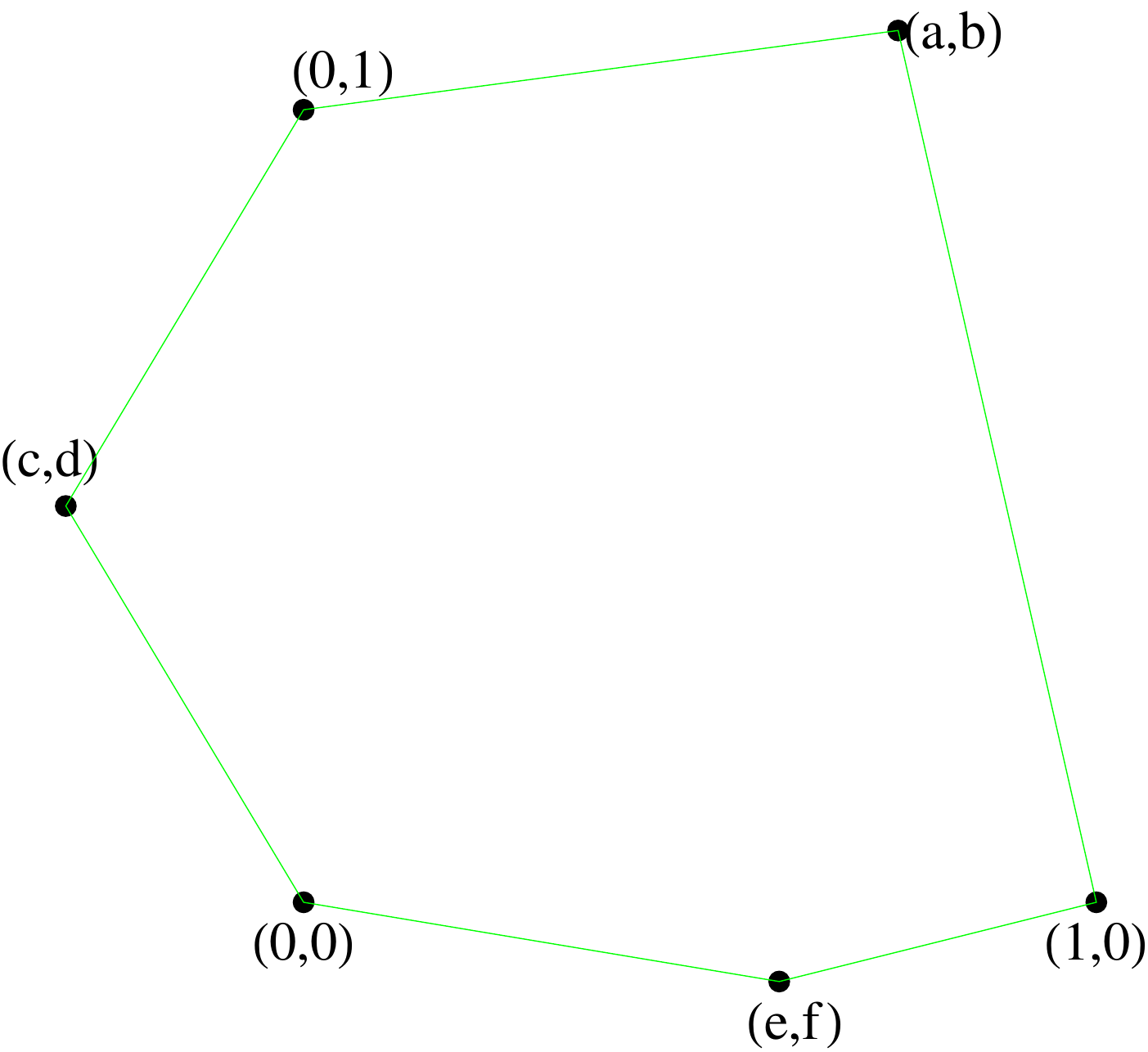}
\end{minipage}
 \hspace{-1.0cm}
\begin{minipage}[c]{0.65\textwidth}
	\includegraphics[width=\linewidth]{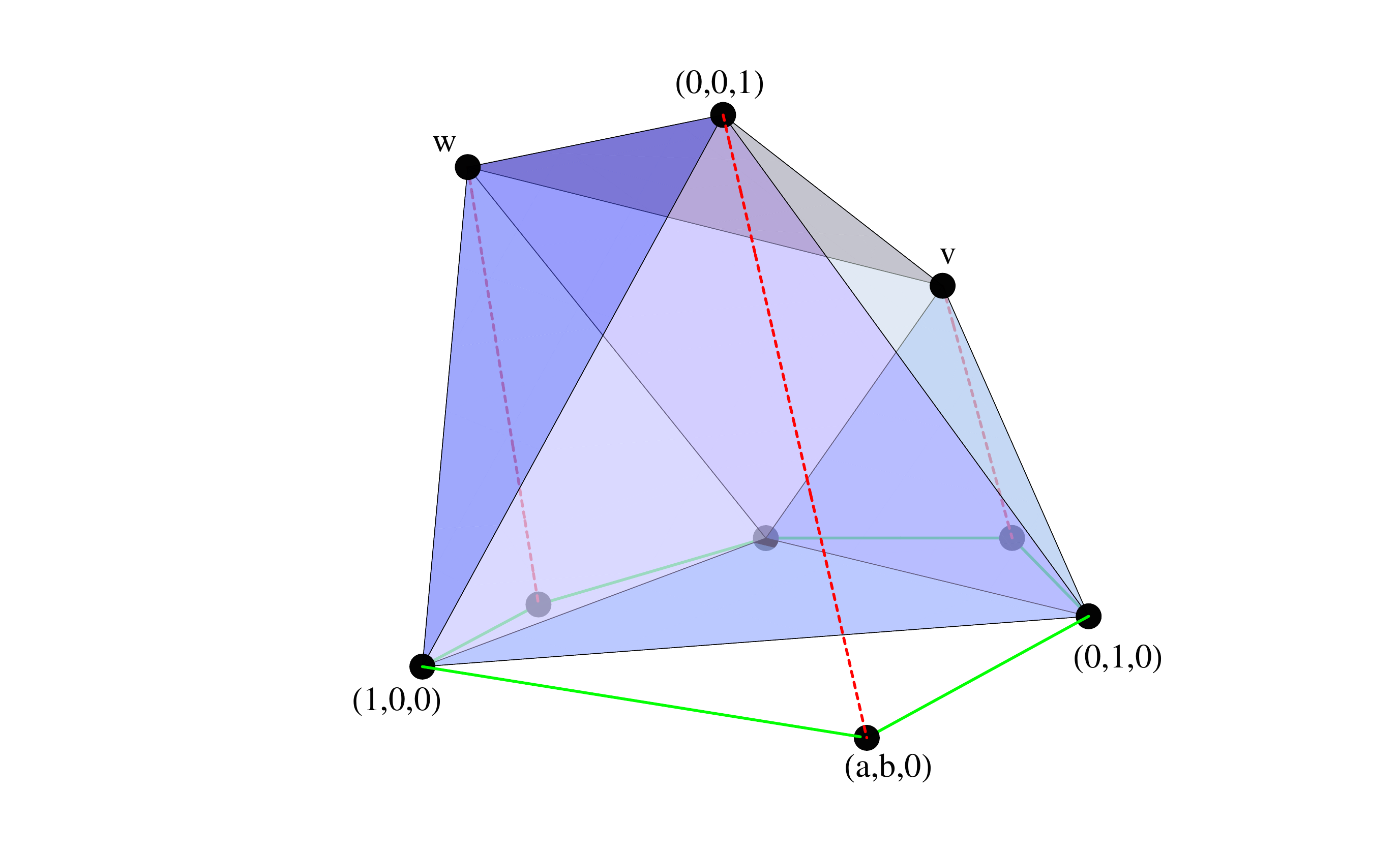}
\end{minipage}
	\caption{This figure depicts the lift of a hexagon to an octahedron shown in the proof of Theorem~\ref{thm:hexagons}.  The left picture shows a hexagon in a normalized form.  The right picture shows the octahedron and its linear projection onto the hexagon.  The projection map is the identity on the xy-plane and is depicted by the red dashed lines for the vertices of the octahedron not in the xy-plane.}
	\label{fig:hexlift}
\end{figure}


To obtain our results for nonnegative matrices of rank three, we recall the notion of a {\em generalized slack matrix} and an interpretation of its psd rank.

\begin{definition} \label{def:general slack}
Let $P \subset \RR^n$ be a full-dimensional polytope and $Q \subset \RR^n$ be a polyhedron with $P \subseteq Q$.  Suppose 
$P$ is represented as a convex hull of points in the form $P = \textup{conv}(p_1,\ldots,p_v)$ and $Q$ is represented by inequalities as $Q = \left\{x \in \RR^n \; | \; c_j^Tx \leq d_j, \,\,j=1,\ldots,f \right\}$ where $c_j \in \RR^n$ and 
$d_j \in \RR$.  Then the generalized slack matrix of the pair $P,Q$ is the $v \times f$ nonnegative matrix $S_{P,Q}$ whose $(i,j)$-entry is $d_j - c_j^T p_i$.
\end{definition}

It is a well-known result in the community that the nonnegative rank of $S_{P,Q}$ is the smallest nonnegative rank of a polyhedron $R$ such that $P \subseteq R \subseteq Q$ (\cite{BraunFioriniPokuttaSteurer}; see \cite{GillisGlineur},\cite{Pashkovich} for related statements). The same result also holds for psd rank, and in fact for any cone rank in the sense of \cite{GPT2011}. We include a proof of the psd case, as we will use the result more than once and it does not seem to be written anywhere.

\begin{proposition}\label{prop:generalized cone extensions}
Let $P$ and $Q$ be polyhedra as in Definition~\ref{def:general slack}, and suppose $Q$ does not contain any lines.  Then $\rankpsd(S_{P,Q})$ is equal to the smallest $k$ such that there exists an affine slice $L$ of $\PSD^k$ and a linear map $\pi$ such that $P \subseteq \pi(L) \subseteq Q$. (We call $k$ the psd rank of the pair $P,Q$ and it measures the smallest possible psd rank of a convex set sandwiched between $P$ and $Q$.)
\end{proposition}

\begin{proof}
After translation and rescaling we may assume that 
$$Q=\left\{x \in \RR^n \; | \; c_j^Tx \leq 1, \,\,j=1,\ldots,f \right\}.$$  Let $\ell$ denote the psd rank of $S_{P,Q}$.  
Then $S_{P,Q}$ has a psd factorization through $\PSD^\ell$.  Thus there exist matrices $U_1,\ldots,U_v,V_1,\ldots,V_f \in \PSD^\ell$ such that $(S_{P,Q})_{ij} = \left\langle U_i, V_j \right\rangle$.  Define an affine set 
\[ A=\left\{(x,M) \in \RR^n \times \mathcal{S}^\ell \; | \; 1 - c_j^Tx =\left\langle M,V_j \right\rangle \textup{ for all } j=1,\ldots,f \right\}. \]
Let $A_M$ be the projection of $A$ onto the $M$ coordinates and define $L = A_M \cap \PSD^\ell$.  Define $\pi$ to be the map on $L$ that sends $M$ to any element $x \in \RR^n$ where $(x,M) \in A$.  This map is well-defined and linear.  Since $(p_i, U_i) \in A$, we see that $P \subseteq \pi(L)$.  Also, for $z \in \pi(L)$, we have that $c_j^Tz \leq 1$ for all $j=1,\ldots, f$.  Thus, $\pi(L) \subseteq Q$.  Hence, $\ell$ is greater than $k$, the psd rank of the pair $P,Q$.

For the converse, note that there exists a convex set $C$ with $P \subseteq C \subseteq Q$ such that $C$ has psd rank $k$. By \cite[Theorem 2.4]{GPT2011}, the slack operator $S_C$ is factorizable through $\PSD^k$, i.e. there exist maps
$\sigma: C \rightarrow \PSD^k$ and $\tau: C^\circ \rightarrow \PSD^k$ 
such that $1- \left\langle x,y \right\rangle = \left\langle \sigma(x),\tau(y) \right\rangle$ for $(x,y) \in C\times C^\circ$.  Here $C^\circ$ denotes the polar of $C$. Then $\sigma(p_1),\ldots,\sigma(p_v),\tau(c_1),\ldots,\tau(c_f)$ give a $\PSD^k$-factorization of $S_{P,Q}$, and so $k \geq \ell$.
\end{proof}

Now suppose we are given a nonnegative $p \times q$ matrix $M$ with $\rank(M)=3$ and we are interested in $\rankpsd(M)$.  First, we may assume that $M$ has no zero rows, since adding or removing zero rows from $M$ will not affect its psd rank.  Therefore, if $\onevec$ denotes the vector of all ones, then $M \onevec$ is a strictly positive vector. Since scaling the rows of $M$ by positive scalars does not affect the psd rank, we can then assume that $\onevec$ 
is in the column span of $M$.  Now consider a {\em rank factorization}  $M=UV$ with $U \in \RR^{p \times 3}$ having rows $U_i=(1,u_i^T)$ for $u_i \in \RR^2$ and $V \in \RR^{3 \times q}$. 
Let $$P :=\textup{conv}(u_1,\ldots,u_p) \textup{ and } Q :=\left\{x \in \RR^2 \; : \; (1, x^T) V \geq 0\right\}.$$ Then the pair $P,Q$ satisfies the conditions of Proposition~\ref{prop:generalized cone extensions} and $M=S_{P,Q}$.  Hence, 
$$\rankpsd(M) = \rankpsd(S_{P,Q}) \leq \rankpsd(P)$$
where the inequality follows from Proposition~\ref{prop:generalized cone extensions}. In particular, a $6 \times q$ nonnegative matrix of rank three is the generalized slack matrix of a hexagon inside a $q$-gon and so has psd rank at most four.
Since $\rankpsd(M) = \rankpsd(M^T)$, the psd rank of a $p \times 6$ nonnegative matrix of rank three is also at most four. 

\begin{theorem} \label{thm:rank 3}
Let $M$ be a nonnegative $p \times q$ matrix with $\rank(M)=3$.  Then $\rankpsd(M) \leq 4\left\lceil \frac{\min\{p,q\} }{6} \right\rceil$. In particular, the psd rank of an $v$-gon is at most $4\left\lceil \frac{v}{6} \right\rceil$.
\end{theorem}
\begin{proof}
We can write $M$ as the concatenation of $\lceil q/6 \rceil$ matrices with $p$ rows and at most six columns, each of which therefore has psd rank at most four. The result now follows by noting that the psd rank of the concatenation of two matrices is at most the sum of the psd ranks of the individual matrices. Indeed, if $\{A_i \}, \{B_j\}$ factorize $M$ and $\{A_i'\}, \{C_k\}$ factorize $M'$ then the following block diagonal matrices factorize $\left[ M_1 \,\,M_2 \right]$:
$$  \left\{ \left( \begin{array}{cc} A_i & 0 \\ 0 & A_i' \end{array} \right) \right\}, \,\, \left\{ \left( \begin{array}{cc} B_j & 0 \\ 0 & 0 \end{array} \right), 
\,\,\left( \begin{array}{cc} 0 & 0 \\ 0 & C_k \end{array} \right) \right\}.$$
\end{proof}

Very little is known about the ranks of $v$-gons for $v \geq 7$. For example, we know that all $7$-gons have psd rank either four or five by Theorem \ref{thm:hexagons}, Lemma \ref{lem:adding a facet} and \cite{GRT2012}, but there is no concrete heptagon whose psd rank is actually known. We know some $8$-gons with psd rank four, but we have no idea how high their psd rank can be, apart from the trivial upper bound of six, obtained again by Lemma \ref{lem:adding a facet}. In fact the smallest ``concrete'' polygons known to have psd rank greater than four are the generic polytopes whose lower bounds are guaranteed by Theorem \ref{thm:generic polytope}, which in this case is a generic $129$-gon.

\section{Geometry of Minimal PSD Rank}

Up until now in this paper we focused on studying matrices of a fixed rank which have high psd rank. In particular, in the previous section we gave upper bounds on the psd rank of a rank three matrix.
In this section we go in the opposite direction and study matrices of fixed rank with minimal psd rank. Given a nonnegative matrix $M$ of rank three, a dimension count immediately shows that $\rankpsd(M) \geq 2$.
We now derive a geometric characterization of when $\rankpsd(M)=2$ 
which will generalize to higher values of rank and yield a complexity result for psd rank.

As before Theorem~\ref{thm:rank 3}, we may assume that our rank three matrix $M$ has size $p \times q$, it has no all-zero rows or columns,
and that $\onevec$ is in the column span of $M$.  Let $M=UV$ be a rank factorization of $M$ with $\onevec$ as the first column of $U$.  Let the rows of $U$ be $(1,u_1^T), \ldots, (1,u_p^T)$ and define polyhedra $P := \conv(u_1,\ldots, u_p)$ and $Q :=\{ x \in \RR^2 \, |\, (1,x^T) V \geq 0 \}$ as before.  Then $P \subseteq Q$ and $M = S_{P,Q}$.

By Proposition~\ref{prop:generalized cone extensions}, we know that $\rankpsd(M)=2$ if and only if there exists a linear map $\pi$ and an affine space $L$ such that $P \subseteq \pi(L \cap \PSD^2) \subseteq Q$.  Since translating $P$ and $Q$ will not affect the slack matrix $M$, we may assume that $0 \in \textup{int}(P)$.  Under this assumption, we see that the affine space $L$ cannot be all of $\mathcal{S}^2$.  Hence, $L$ must be a two-dimensional slice of $\mathcal{S}^2$ and $\pi|_L$ must be invertible.  Since $\PSD^2$ is linearly equivalent to the positive half of the three-dimensional second order cone, $\{ (x,y,z) \; | \; x^2+y^2 \leq z^2 \textup{ and } z \geq 0\}$, we see that $L \cap \PSD^2$ is the linear image of the convex hull of a ``half-conic'' where half-conics are all ellipses, parabolas, and connected components of hyperbolas in $\RR^2$.  Finally, we use the fact that the set of conics is invariant under invertible linear transformations to see the following.

\begin{proposition}\label{prop:conic characterization}
Let $M$ be a nonnegative rank three matrix.  Let $P \subseteq Q \subseteq \RR^2$ be the polytope and polyhedron arising from a rank factorization of $M$ as above.  Then $\rankpsd(M)=2$ if and only if there exists a half-conic such that its convex hull $C$ satisfies $P \subseteq C \subseteq Q$.  In particular if $Q$ is bounded, then $\rankpsd(M)=2$ if and only if we can fit an ellipse between $P$ and $Q$.
\end{proposition}

\begin{example}
Consider the one-parameter family of matrices
$$M_{\varepsilon}=\left[
\begin{array}{cccc}
2- \varepsilon & 2- \varepsilon & \varepsilon &  \varepsilon \\
 \varepsilon & 2- \varepsilon & 2- \varepsilon &  \varepsilon \\
 \varepsilon &  \varepsilon & 2- \varepsilon & 2- \varepsilon \\
2- \varepsilon &  \varepsilon & \varepsilon & 2- \varepsilon 
\end{array}
\right],$$
with $\varepsilon \in [0,1]$. For $\varepsilon \not = 1$ this matrix has rank $3$, and we would like to know for which (if any) values of $\varepsilon$ we get $\rankpsd(M)=2$. 
Note that $M_{\varepsilon}=S_{(1-\varepsilon)P,P}$,
where $P$ is the $\pm 1$ square. It is easy to see that we can put a half-conic between $(1-\varepsilon)P$ and $P$ if and only if $1-\varepsilon \leq \sqrt{2}/2$ as seen in Figure \ref{Fig:nestedsquares}. Since it is known that the square itself has psd rank three, Proposition \ref{prop:generalized cone extensions} allows us to completely determine the psd ranks of this matrix family:
$$\rankpsd{M_{\varepsilon}=\left\{ 
\begin{array}{ll}
1 & \textrm{ if } \varepsilon=1;\\
2 & \textrm{ if } \varepsilon \in [1-\sqrt{2}/2,1) ;\\
3 & \textrm{ if } \varepsilon \in [0,1-\sqrt{2}/2).
\end{array}
\right.}$$

\begin{figure}
\centering
\captionsetup{justification=centering}
\includegraphics[scale=0.2]{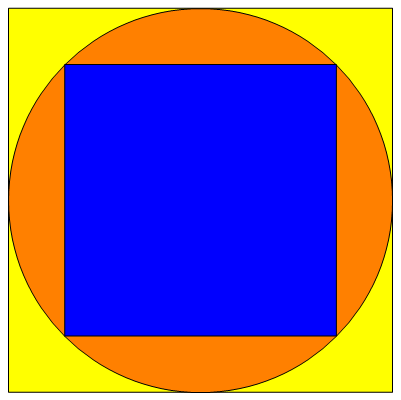}
\caption{Disk nested between $P$ and $(\sqrt{2}/2) P$ where $P$ is the unit square.} 
\label{Fig:nestedsquares}
\end{figure}
 
\end{example}

The geometric techniques used above generalize to higher rank matrices.  Let $M \in \RR^{p \times q}$ be a nonnegative matrix of rank $d= {k+1 \choose 2}$.  Then by a dimension count, $\rankpsd(M) \geq k$.  Thus we can ask the following decision 
problem about $M$:

\begin{definition}
{MIN PSD RANK}: Given a nonnegative matrix $M$ of rank ${k+1 \choose 2}$, is $\rankpsd(M) = k$?
\end{definition}

For ease in working with higher dimensions, we will switch from the polytope viewpoint used above to a conic viewpoint.  In the remainder of this section $d={k+1 \choose 2}=\rank(M)$.  Let $M=UV$ be a rank factorization and let $P,Q$ be the cones $P = \cone(u_1,\ldots,u_p)$ and $Q=\{ x \in \RR^d \, |\, x^T V \geq 0 \}$ where $u_i$ are the rows of $U$.  Then $P$ and $Q$ are $d$-dimensional cones 
with $P \subseteq Q$ and $M=S_{P,Q}$ where $S_{P,Q}$ is a generalized slack matrix of the pair of cones $P,Q$, defined analogously to that for pairs of polyhedra. Using Proposition~\ref{prop:generalized cone extensions} 
and counting dimensions, we get the following geometric characterization of the MIN PSD RANK problem:

\begin{proposition}\label{prop:geometric condition} The psd rank of $M$ is $k$ if and only if there is an invertible linear map $\pi \,:\, \mathcal{S}^k \rightarrow \RR^d$ such that $P \subseteq \pi(\PSD^k) \subseteq Q$.
\end{proposition}

In \cite{Vavasis}, Vavasis defined EXACT NMF ({\em Nonnegative Matrix Factorization}) as the problem of determining whether the nonnegative rank of a given matrix $M$ equals its rank. He also defined INTERMEDIATE SIMPLEX which asks, 
given two nested polyhedra $P \subseteq Q$, if there is a simplex $T$ such that $P \subseteq T \subseteq Q$. He proceeded to show that EXACT NMF is equivalent to INTERMEDIATE SIMPLEX.  The above reduction of MIN PSD RANK to the geometric condition of Proposition~\ref{prop:geometric condition} can be thought of as the psd analog to the equivalence shown by Vavasis.

Now we will reduce the geometric criterion into a semialgebraic set feasibility problem.  Consider the basis of $\mathcal{S}^k$ given by the elementary symmetric matrices $E_{ij}$ defined as follows.  Let $E_{ii}$ be the matrix with a one in position $(i,i)$ and zeros everywhere else.  For $i<j$, let $E_{ij}$ be the matrix with $\frac{1}{\sqrt{2}}$ in positions $(i,j)$ and $(j,i)$ and zeros everywhere else. This basis allows a natural bijection between $\mathcal{S}^k$ and $\RR^d$ by identifying a symmetric matrix $Y = \sum_{1 \leq i \leq j \leq d} E_{ij} y_{ij}$ with the vector $y = (y_{ij}) \in \RR^d$. Note that this bijection preserves the inner product in $\mathcal{S}^k$ (this is the reason for the $\sqrt{2}$ factors).  Let $L$ be the $r \times r$ nonsingular matrix representing the invertible linear map $\pi$ with respect to the above basis. Then $\pi(Y) = Ly$, and $\pi^{-1} \,:\, \RR^d \rightarrow \mathcal{S}^k$ sends $z \mapsto L^{-1}z$ where $L^{-1}z$ corresponds to a matrix in $\mathcal{S}^k$ under the bijection discussed above.
We can now write down the conditions given by Proposition~\ref{prop:geometric condition} in terms of $L$ and $L^{-1}$.

The condition that $P \subseteq \pi(\PSD^k)$ is equivalent to $\pi^{-1}(u_i) \in \PSD^k$ for every generator $u_i$ of $P$. Thus we need $L^{-1}u_i \succeq 0$ for each row $u_i$ of $U$. Note that each entry in the symmetric matrix corresponding to $L^{-1}u_i$ is a linear polynomial in the entries of $L^{-1}$. The condition $\pi(\PSD^k) \subseteq Q$ says that for each column $v_j$ of $V$, the linear inequality $v_j^Tx \geq 0$ is valid on $\pi(\PSD^k)$, or equivalently, that for every matrix $A \in \PSD^k$, $v_j^T(\pi(A)) \geq 0$. Therefore, we get that for every column $v_j$ of $V$, the 
symmetric matrix corresponding to $v_j^T L$ is psd. Putting all this together we get the following reduction of the MIN PSD RANK problem.

\begin{proposition} \label{prop:reduction to psd conditions}
The matrix $M$ has psd rank $k$ if and only if there are two matrices $L,K \in \RR^{d \times d}$ such that 
\begin{enumerate}
\item $L$ is the inverse of $K$, i.e., $LK=KL=I$,
\item The $k \times k$ linear matrix inequality $Ku_i \succeq 0$ holds for each row $u_i$ of $U$, 
\item The $k \times k$ linear matrix inequality $v_j^T L\succeq 0$ holds for each column $v_j$ of $V$.
\end{enumerate}
Further, the above system can be written down in polynomial time from $M$.
\end{proposition}

\begin{proof} The equivalence of MIN PSD RANK and the feasibility of the above system was argued in the discussion before the proposition. The scalars in the system come from a rank factorization of $M$ which can be done in polynomial time.
\end{proof}

The number of variables in the above semialgebraic system depends only on $k$ and not on the size of the input matrix $M$.  In \cite{Renegar}, Renegar showed that the feasibility of a system of $m$ polynomial inequalities and equalities in $\ell$ variables with degree at most $j$ can be determined in time $(m j )^{O(\ell)}$.  Here, Renegar used the Blum-Shub-Smale model of complexity for computing with real numbers, so the only requirement on the coefficients of the polynomials is that they are real numbers.  We use this to get a complexity result for MIN PSD RANK.

\begin{theorem}\label{thm:running time}
Using the Blum-Shub-Smale model of complexity, the problem MIN PSD RANK can be solved in time $(pq)^{O(d^{2.5})}$ where $p \times q$ is the dimension of the input matrix $M$ and $d={k+1 \choose 2}$ is the rank of $M$.  In particular, for fixed rank, the problem MIN PSD RANK can be solved in polynomial time.
\end{theorem}

\begin{proof}
First, we consider the problem formulated in Proposition~\ref{prop:reduction to psd conditions}.  This problem can be formulated as the existence of a solution to a system of $d^2+2^k(p+q)$ polynomial equalities and inequalities in $2d^2$ variables with each polynomial having degree less than or equal to $k$.  By applying \cite{Renegar} and noting that $d\sim k^2$, we see that this problem can be solved in time $(pq)^{O(d^{2.5})}$.  We conclude by noting that MIN PSD RANK can be reduced to the above problem in time polynomial in $pq$.
\end{proof}

In \cite{Vavasis}, Vavasis showed that EXACT NMF is NP-Hard.  The corresponding question for MIN PSD RANK is still open.  We can consider the more general problem: given a nonnegative $p \times q$ matrix $M$ and a number $k$, determine if $\rankpsd(M) \leq k$.  For the analogous problem with nonnegative rank, Moitra \cite{Moitra} showed an algorithm that runs in time $(pq)^{O(k^2)}$.  Theorem~\ref{thm:running time} can be seen as a restricted psd analog of Moitra's result.

\bibliographystyle{plain}

\end{document}